\documentclass[reqno]{amsart}
\usepackage{amsthm,amsmath,amssymb}
\usepackage[pdftex]{graphicx}
\usepackage{braket}
\usepackage{color}
\usepackage{cleveref}
\usepackage{amscd}
\numberwithin{equation}{section}

\title{An unoriented analogue of slice-torus invariant}
\author{Kouki Sato}
\date{}
\address{\textsc{Meijo University,Tempaku, Nagoya 468-8502, Japan}}
\email{satokou@meijo-u.ac.jp}

\newtheorem{thm}{Theorem}[section]
\newtheorem{prop}[thm]{Proposition}
\newtheorem{lem}[thm]{Lemma}
\newtheorem{cor}[thm]{Corollary}
\newtheorem{claim}{Claim}

\theoremstyle{definition}

\newtheorem{dfn}[thm]{Definition}

\newtheorem{remark}[thm]{Remark}
\newtheorem*{acknowledge}{Acknowledgements}

\DeclareMathOperator{\Top}{Top}

\newcommand{\R}{\mathbb R}
\newcommand{\Z}{\mathbb Z}

\newcommand{\mC}{\mathcal C}

\begin{document}

\begin{abstract}
A slice-torus invariant is an $\R$-valued homomorphism on the knot concordance group whose value gives a lower bound for the 4-genus such that the equality holds for any positive torus knot. Such invariants have been discovered in many of knot homology theories, while it is known that any slice-torus invariant does not factor through the topological concordance group. 
In this paper, we introduce the notion of {\it unoriented slice-torus invariant},
which can be regarded as the same as slice-torus invariant except for the condition about the orientability of surfaces.
Then we show that the Ozsv\'ath-Stipsicz-Szab\'o $\upsilon$-invariant,
the Ballinger $t$-invariant and the Daemi-Scaduto $h_{\Z}$-invariant (shifted by a half of the knot signature) are unoriented slice-torus invariants.
As an application, we give a new method for computing the above invariants, which is analogous to Livingston's method for computing slice-torus invariants.
Moreover, we use the method to prove that any unoriented slice-torus invariant does not factor through the topological concordance group. 
\end{abstract}

\maketitle

\section{Introduction}
For the last two decades, a number of real-valued homomorphisms on the knot concordance group have been introduced from various theories, while many of them are difficult to compute in general. The study of {\it slice-torus invariants} has enable us to compute some of them for a large family of knots and compare the properties of them.
\begin{dfn}[\cite{Lewark:2014, Livingston:2004}]
\label{dfn:slice-torus-inv}
    A \textit{slice-torus invariant} $\nu$ is an abelian group homomorphism
    \[
        \nu: \mC \rightarrow \R
    \]
    satisfying the following two conditions: 
    \begin{enumerate}
        \item[(Slice)] \ $|\nu(K)| \leq g_4(K)$ for any knot $K$,
        \item[(Torus)] \ $\nu(T_{p, q}) = g_4(T_{p,q}) = \frac{(p - 1)(q - 1)}{2}$ for any positive $(p, q)$-torus knot $T_{p, q}$.
    \end{enumerate}
    Here $\mC$ denotes the smooth concordance group of knots in $S^3$, and $g_4(K)$ the 4-genus of a knot $K$. 
\end{dfn}

Examples of slice-torus invariants are (i) the $\tau$-invariant obtained from knot Floer homology \cite{OS:2003}, (ii) the Rasmussen $s$-invariant and its generalization over any field $F$ using Bar-Natan homology \cite{LS:2014_rasmussen, MTV:2007,  Rasmussen:2010}, (iii) the $s_n$-invariants $(n \geq 2)$ from $\mathfrak{sl}_n$ Khovanov--Rozansky homologies \cite{Lobb:2009,Lobb:2012,Wu:2009}, (iv) the $\tau^\#$-invariant from framed instanton Floer homology \cite{Baldwin-Sivek:2021}, (v) the $\tilde{s}$-invariant from equivariant singular instanton Floer homology \cite{DISST:2022}, (vi) the $\widetilde{ss}_c$-invariants from reduced Khovanov homologies \cite{Sano:2020, SS:2022},
and (vii) the $q_M$-invariant from equivariant Seiberg-Witten theory \cite{IT:2024}. More studies on  slice-torus invariants are given in \cite{Cavallo:2020,Feller:2022}.

On the other hand,  it has been also shown that several concordance homomorphisms $f \colon \mC \to \R$ satisfy
\[
\left|f(K) - \frac{e(F)}{4} \right| \leq \frac{b_1(F)}{2}
\]
where $F$ is a possibly non-orientable surface in $B^4$ with boundary $K$, and
$e(F)$ and $b_1(F)$ the normal Euler number and the first Betti number of $F$, respectively.
In particular, combining it with the Gordon-Litherland inequality $|\sigma(K) - e(F)/2| \leq b_1(F)$ in \cite{GL:1978}, we have the bound
\[
\left|f(K) - \frac{\sigma(K)}{2} \right| \leq \gamma_4(K)
\]
where $\gamma_4(K)$ is the {\it non-orientable 4-genus}, 
i.e.\ the minimal first Betti number of surfaces in $B^4$ with boundary $K$.
Examples with such a property are (i) the Ozsv\'ath-Stipsicz-Szab\'o $\upsilon$-invariant where $\upsilon = \Upsilon(1)$ \cite{OSS:2017},
(ii) the Kronheimer-Mrowka $f_{\sigma}$-invariants with condition $\sigma(T_0)=1$ \cite{KM:2021}, 
(iii) the Ballinger $t$-invariant \cite{Ballinger:2020}, and (iv) the Daemi-Scaduto $h_{\Z}$-invariant \cite{DS:2020}. 
In particular, for any positive torus knot $T_{p,q}$, we have
\[
\upsilon(T_{p,q}) - \frac{\sigma(T_{p,q})}{2} = -\frac{t(T_{p,q})}{2} -\frac{\sigma(T_{p,q})}{2} = h_{\Z}(T_{p,q}).
\]
Hence, it is natural to ask whether the above value for $T_{p,q}$ can be characterized using some geometric notion, like $\nu(T_{p,q})=g_4(T_{p,q})$ for slice-torus invariants.  
(Remark that $\upsilon(T_{2,q})-\sigma(T_{2,q})/2 = 0 < 1=\gamma_4(T_{2,q})$, and hence $\gamma_4$ is not suitable.)
In this paper, we give an affirmative answer to this question. To state the main theorem, we fix some notations.
For a surface $F$ properly embedded in $B^4$, let $M_F$ denote 
the double branched cover of $B^4$ over $F$,
and $b_2^-(M_F)$ the number of negative eigenvalues of the intersection form of $M_F$.
Then we can define the knot concordance invariant
\[
b^{-}_*(K) :=
\min \left\{b^-_2(M_F) \mid F \text{ is a surface in $B^4$ with $\partial F = K$} \right\},
\]
which is introduced in \cite{Sato:2019}.
\begin{thm}
\label{thm:main}
For any positive torus knot $T_{p,q}$, we have 
\[
\upsilon(T_{p,q}) - \frac{\sigma(T_{p,q})}{2} = -\frac{t(T_{p,q})}{2} -\frac{\sigma(T_{p,q})}{2} = h_{\Z}(T_{p,q}) = b^-_*(T_{p,q}).
\]
Moreover, the value $b^-_*(T_{p,q})$ is realized by Batson's non-orientable surface in \cite{Batson:2014}.
\end{thm}
In light of \Cref{thm:main}, we define the notion of {\it unoriented slice-torus invariant} as follows.
\begin{dfn}
\label{dfn:UST}
    An \textit{unoriented slice-torus invariant} (or simply {\it UST-invariant}) $f$ is an abelian group homomorphism
    \[
        f: \mC \rightarrow \R
    \]
    satisfying the following two conditions: 
    \begin{enumerate}
        \item[(Slice)] \ $|f(K) - \frac{e(F)}{4}| \leq \frac{b_1(F)}{2}$ for any surface $F$ in $B^4$ with boundary $K$,
        \item[(Torus)] \ $f(T_{p, q}) - \frac{\sigma(T_{p,q})}{2} = b^-_*(T_{p,q})$ for any positive $(p, q)$-torus knot $T_{p, q}$.
    \end{enumerate}
\end{dfn}
\begin{remark}
\label{rem}
Since $\sigma(K)=\sigma(M_F)$ and $g(F) = \frac{b_2(M_F)}{2}$ for any orientable surface $F$ in $B^4$ with boundary $K$,
the condition (Torus) in \Cref{dfn:slice-torus-inv} is equivalent to
\[
\nu(T_{p,q}) -\frac{\sigma(T_{p,q})}{2} = 
\min \left\{b^-_2(M_F) \mid F :\text{orientable surface in $B^4$ with $\partial F = T_{p,q}$} \right\}.
\]
Therefore, \Cref{dfn:UST} can be regarded as the same as \Cref{dfn:slice-torus-inv} except for the condition about the orientability of surfaces.

We also note that the condition (Slice) in \Cref{dfn:UST} is equivalent to 
\begin{align}
\label{eq:UST}
f(K) - \frac{\sigma(K)}{2} \leq b^-_*(K),
\end{align}
which follows from the Gordon-Litherland equality
$\sigma(K) - e(F)/2 = \sigma(M_F)$ in \cite{GL:1978}.
\end{remark}

It follows from the above arguments that $\upsilon$, $-t/2$ and $h_{\Z} + \sigma/2$ are concrete examples of UST-invariants.

For a given UST-invariant $f$ and knot $K$, we call a surface $F \subset B^4$ 
{\it a realizing surface for $(f, K)$} if $\partial F = K$ and $f(K)- \frac{\sigma(K)}{2} = b^-_2(M_F)$. 
(Note that such a surface does not exist for a general knot. For instance, if $f(K)-\frac{\sigma(K)}{2} > 0$, then $f(K^*)-\frac{\sigma(K^*)}{2} < 0$ for the orientation reversed mirror $K^*$ of $K$, while $b^-_2(M_F)$ is always non-negative.)
Then, we can see that any subsurface of a realizing surface is also realizing. This is an analogous observation to Livingston's method \cite{Livingston:2004}.

\begin{thm}
\label{thm:subsurface}
For a UST-invariant $f$ and knot $K$, 
let $F \subset B^4$ be a realizing surface for $(f,K)$.
If there exists a 4-ball $B \subset B^4$ such that $\partial B$ is transverse to $F$ and $K':=F \cap \partial B$ is a knot,  then $F' := F \cap B$ is a realizing surface for $(f,K')$.
\end{thm}

One advantage of UST-invariant is that a realizing surface $F$ for $(f,K)$ admits positive stabilizations. That is, even if we take the boundary connected sum of $F$ and $F' \subset B^4$ such that $\partial F'$ is the unknot and $M_{F'}$ is positive definite, the resulting surface is still realizing for $(f,K)$. As a practical application of this advantage, 
we have the following proposition. Here, a {\it ribbon surface} is a self-intersecting surface in $S^3$ with only ribbon singularities shown in \Cref{fig:ribbon}, 
and a {\it positive full-twist} of a ribbon surface is a local deformation shown in 
\Cref{fig:positive-twist}.  We also regard a ribbon surface $F$ as properly embedded in $B^4$ in the usual way (i.e.\ pushing the interior of $F$ into the interior of $B^4$ so that the restriction of the function $(x,y,z,w) \mapsto \sqrt{x^2+y^2+z^2+w^2}$ to $F$ is a Morse function with no local maxima).

\begin{prop}
\label{prop:positive-twist}
For a UST-invariant $f$ and knot $K$, 
let $F$ be a ribbon surface which is a realizing surface for $(f,K)$. 
If a ribbon surface $F'$ with $\partial F' = K'$ is obtained from $F$ by a positive full-twist, then $F'$ is a realizing surface for $(f,K')$.
\end{prop}

\begin{figure}[tbp]
\includegraphics[scale = 0.8]{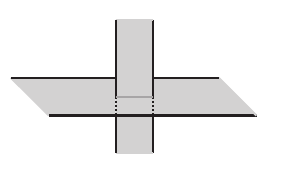}
\caption{\label{fig:ribbon} Ribbon singularity}
\end{figure}
\begin{figure}[tbp]
\includegraphics[scale = 0.8]{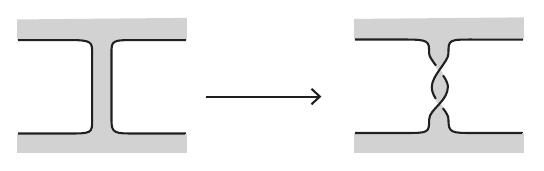}
\caption{\label{fig:positive-twist} positive full-twist}
\end{figure}

Here we show an application of \Cref{prop:positive-twist}.
Consider the pretzel knot $P(-2,p,q)$ for odd integers $p,q \geq 3$, shown in 
\Cref{fig:pretzel}. Note that $P(-2,3,3)$ is the torus knot $T_{3,4}$, and it is easy to check that a surface in $S^3$ obtained by the checkerboard coloring of \Cref{fig:pretzel} for $p=q=3$
is a realizing surface for $(f,T_{3,4})$, where $f$ is any UST-invariant.  
Now, it immediately follows from \Cref{prop:positive-twist} that 
a surface obtained by the checkerboard coloring of \Cref{fig:pretzel} for
any odd $p,q \geq 3$ is a realizing surface for $(f,P(-2,p,q))$. 
As a corollary, we have:

\begin{cor}
For any UST-invariant $f$ and odd integers $p,q \geq 3$, we have
\[
f(P(-2,p,q)) = \frac{2-p-q}{2}.
\]
\end{cor}

\begin{figure}[tbp]
\includegraphics[scale=0.8]{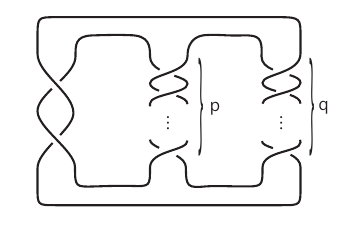}
\caption{\label{fig:pretzel} The pretzel knot $P(-2,p,q)$}
\end{figure}

As another application of \Cref{thm:subsurface}, 
we obtain the following result about the topological concordance group $\mC_{\Top}$.

\begin{thm}
\label{thm:topslice}
Any unoriented slice-torus invariant does not factor through $\mC_{\Top}$.
\end{thm}

Concretely, we will show that for the negatively-clasped untwisted Whitehead double of the knot $5_2$ in Rolfsen's table \cite{Rolfsen:1972} (denoted by $D_-(5_2)$), its standard Seifert surface $F$ can be embedded in a realizing surface for $(f,T_{3,4}\#5_2^*)$. 
Then, \Cref{thm:subsurface} implies that $f(D_-(5_2)) = b^-_2(M_F) + \frac{\sigma(D_-(5_2))}{2}= 1 \neq 0$.

As a consequence of \Cref{thm:topslice}, we can conclude that UST-invariants cannot be obtained from any topological concordance theory as reviewed in \cite{Livingston:2005} (like algebraic concordance theory \cite{Levine:1969Inv, Levine:1969Knot}, Casson-Gordon theory \cite{CG:1978, CG:1986}, Cochran-Orr-Teichner theory \cite{COT:2003}, and etc.). 

\begin{acknowledge}
The author would like to thank Aliakbar Daemi, Hayato Imori 
and Masaki Taniguchi for their stimulating conversations.
\end{acknowledge}

\section{Preliminaries}

\subsection{Non-orientable cobordisms obtained by $H(2)$-moves}

The {\it $H(2)$-move} is a deformation of link diagram shown in \Cref{fig:H(2)}.
If a diagram $D_1$ for a knot $K_1$ is deformed into a diagram $D_2$ for a knot $K_2$ by an $H(2)$-move,
then it gives rise to a non-orientable cobordism $C$ from $K_1$ to $K_2$ in $S^3 \times [0,1]$
such that $b_2(M_C)=1$. The computation of $\sigma(M_C)$ is given by the following formula. (Here $w(D_i)$ denotes the writhe of $D_i$.)
\begin{lem}[\text{\cite[Lemma 6.6]{Sato:2019}}]
\label{lem:signature}
$
\sigma(M_C)= \big(\sigma(K_2) -\sigma(K_1) \big) + \frac{1}{2} \big(w(D_2) -w(D_1) \big).
$
\end{lem}

\begin{figure}[tbp]
\includegraphics[scale=0.7]{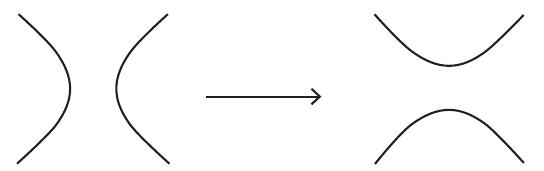}
\caption{\label{fig:H(2)} $H(2)$-move}
\end{figure}

\subsection{Batson's non-orientable surface}
Let $p$ and $q$ be coprime integers with $p>q>0$ and $D_{p,q}$ the standard diagram of $T_{p,q}$ (see the left-hand side of \Cref{fig:torus-knot}). Performing an $H(2)$-move to two adjacent strands in $D_{p,q}$, we obtain a new diagram $D'_{p,q}$ (see the right-hand side of 
\Cref{fig:torus-knot}).
Here we note that $D'_{p,q}$ is a diagram of $T_{p-2t, q-2h}$, where $0<t<p$ and $0<h<q$ are the integers determined by
\[
t \equiv -q^{-1} \  (\mathrm{mod} \  p)
\quad \text{and} \quad
h \equiv p^{-1} \ (\mathrm{mod} \ q).
\]
We also mention that the above $(t,h)$ is a unique pair satisfying $0<t<p$, $0<h<q$ and $hp- tq =1$.

\begin{figure}[tbp]
\includegraphics[scale=0.95]{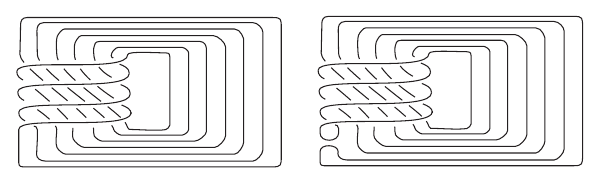}
\caption{\label{fig:torus-knot} The diagrams $D_{p,q}$ and $D'_{p,q}$ for $(p,q)=(7,4)$}
\end{figure}

Since an $H(2)$-move is an involution, we also have an $H(2)$-move deforming $D'_{p,q}$
to $D_{p,q}$, and it gives rise to a non-orientable cobordism
$C_{p,q}$ from $T_{p',q'}$ to $T_{p,q}$, 
where $p':=\max\{|p-2t|, |q-2h|\}$ and $q' := \min \{|p-2t|, |q-2h|\}$.
($T_{p-2t,q-2h}$ is also a positive torus knot, and hence we may identify $T_{p-2t,q-2h}$ with $T_{|p-2t|,|q-2h|}$.)
Repeating this process, we have the sequence 
\[
\begin{CD}
U=T_{p_0,q_0} @>{C_{p_1, q_1}}>>  T_{p_1,q_1}
@>{C_{p_2,q_2}}>> \cdots  @>{C_{p_n,q_n}}>> T_{p_n,q_n} = T_{p,q}
\end{CD}
\]
of non-orientable cobordisms, where $U$ denotes the unknot. Let $C_{p_0,q_0}$ be a standard disk in $B^4$ with boundary $U=T_{p_0,q_0}$. Then, {\it Batson's non-orientable surface} $F_{p,q}$ is given by the composition
\[
F_{p,q} = \bigcup_{i=0}^nC_{p_i,q_i}.
\]
In particular, the double branched cover $M_{F_{p,q}}$ is also regarded as
the composition of $M_{C_{p_i,q_i}}$, and we have
\[
b^-_2(M_{F_{p,q}}) = \sum_{i=0}^n b^-_2(M_{C_{p_i,q_i}}) 
= \sum_{i=1}^n  \tfrac{1}{2}\left(1 - \sigma(M_{C_{p_i,q_i}})\right).
\]
Hence, in light of \Cref{lem:signature}, we see that the following lemma seems to be useful for computing $b^-_2(M_{F_{p,q}})$.

\begin{lem} \label{lem:writhe}
$\frac{1}{2}\left(w(D_{p,q}) - w(D'_{p,q})\right) = h(p-t) + t(q-h) $.
\end{lem}
\begin{proof}
The left-hand side of the desired equality is equal to how many crossings the bold path shown in \Cref{fig:torus-knot2} passes through exactly once. 
Moreover, we may regard the bold path as the union of strands in the $p$-stranded braid 
$\Delta^q = (\sigma_1\sigma_2\cdots\sigma_{p-1})^q$  whose initial points constitute the set 
\[
S := \{1 \leq k \leq p \mid 0\leq \exists i < t, \quad k \equiv 1 -i q \  (\mathrm{mod} \  p)\}.
\]
For example, if $(p,q)=(7,4)$, then $S = \{1, 4, 7, 3, 6 \}$.

Next, we note that each strand in $\Delta^q$ with initial point $1\leq k\leq q$ passes through all crossings in the $k$-th $\Delta$, while any other strand passes through exactly one crossing in the $k$-th $\Delta$. Therefore, the crossings in $\Delta^q$ can be labeled by the pairs $(k,l)$ with $1 \leq k \leq q$, $1 \leq l \leq p$ and $l \neq k$. (We denote by $\mathcal{L}$ the set of these labels.)
Now, let $S_{\leq q} := S \cap \{1 \leq k \leq q\}$, and then we can compute the left-hand side of the desired equality by
\begin{align*}
&\#\{ (k,l) \in \mathcal{L} \mid k \in S_{\leq q} \text{ and } l \notin S \} + 
\#\{ (k,l) \in \mathcal{L} \mid k \notin S_{\leq q} \text{ and } l \in S \} \\[2mm]
&=|S_{\leq q}| (p-|S|) + (q - |S_{\leq q}|)|S|.
\end{align*}
Obviously, we have $|S|=t$. Hence, we only need to compute $|S_{\leq q}|$.
\begin{claim}
$|S_{\leq q}| = \frac{tq+1}{p}$.
\end{claim}
\begin{proof}
Let $m := \frac{tq+1}{p}$ and $[m] := \{j \in \Z \mid 1 \leq j \leq m\}$.
Then, for any $j \in [m]$, there exists a unique integer $i$ with $0 \leq i < t$
such that
\[
mp -(i+1)q +1 \leq jp < mp -iq +1.
\]
(Remark that $p \geq 2$.) In particular, we have $1 \leq (m-j)p-iq +1 \leq q$, and this gives a map $f \colon [m] \to S_{\leq q}$.

Conversely, for any $k \in S_{\leq q}$, there exist unique integers $i$ and $l$ with $0 \leq i < t$ and $l \geq 0$ such that $k = lp + 1 -iq$. Moreover, if we assume $l \geq m$, then
\[
k \geq mp + 1 -iq = (t-i)q+2 \geq q+2 >q, 
\]
which leads to a contradiction. Hence $m-l \in [m]$, and this gives a map 
$g \colon S_{\leq q} \to [m]$. 

It is easy to check that $g$ is the inverse of $f$, and hence $f$ is bijective.
\end{proof}
Now, since $1 \leq m = |S_{\leq q}| \leq q$ and $mp = tq +1 \equiv 1 \  (\mathrm{mod} \  q)$, 
we have 
\[
|S_{\leq q}| = m = h.
\]
This completes the proof.
\end{proof}

\begin{figure}[tbp]
\includegraphics[scale=0.95]{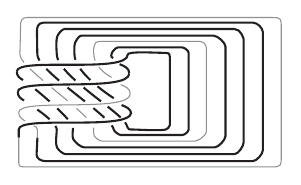}
\caption{\label{fig:torus-knot2} The bold path for $(p,q) =(7,4)$}
\end{figure}

\section{Proof of \Cref{thm:main}}

Now we prove \Cref{thm:main}.

\setcounter{section}{1}
\setcounter{thm}{1}

\begin{thm}
For any positive torus knot $T_{p,q}$, we have 
\[
\upsilon(T_{p,q}) - \frac{\sigma(T_{p,q})}{2} = -\frac{t(T_{p,q})}{2} -\frac{\sigma(T_{p,q})}{2} = h_{\Z}(T_{p,q}) = b^-_*(T_{p,q}).
\]
Moreover, the value $b^-_*(T_{p,q})$ is realized by Batson's non-orientable surface in \cite{Batson:2014}.
\end{thm}

\setcounter{section}{3}
\setcounter{thm}{0}

Here, since the first two equalities are already proved in \cite{Ballinger:2020, DS:2020}, 
we only need to prove the equality
\[
\upsilon (T_{p,q}) - \frac{\sigma(T_{p,q})}{2} = b^-_2({M_{F_{p,q}}}).
\]
In addition, since all of $\upsilon, \sigma$ and $b^-_2(M_{\bullet})$ are additive, 
the above equality follows from  
the equality
\[
\upsilon (\partial C_{p,q}) - \frac{\sigma(\partial C_{p,q})}{2} 
 = b^-_2(M_{C_{p,q}}).
\]
Moreover, by \Cref{lem:signature} and \Cref{lem:writhe}, we see that
\begin{align*}
b^-_2(M_{C_{p,q}}) &= \tfrac{1}{2}\big(1 - \sigma(M_{C_{p,q}}) \big) \\[2mm]
&= \tfrac{1}{2} - \frac{\sigma(\partial C_{p,q})}{2} - \tfrac{1}{4} \big(w(D_{p,q}) - w(D'_{p,q}) \big)
\\[2mm]
&= \tfrac{1}{2} \big( 1 - h(p-t) - t(q-h) \big) - \frac{\sigma(\partial C_{p,q})}{2}.
\end{align*}
Therefore, to prove \Cref{thm:main}, it suffices to prove the following proposition.

\begin{prop}
$
\upsilon (T_{p,q}) - \upsilon(T_{p-2t, q-2h}) = \frac{1}{2} \big( 1 - h(p-t) - t(q-h) \big).
$
\end{prop}
\begin{proof}

For the $\upsilon$-invariant of torus knots, we have Feller-Krcatovich's recursive formula 
\cite{FK:2017}:
\[
\upsilon (T_{p,q}) =
\begin{cases}
\upsilon(T_{p-q,q}) - \frac{1}{4}q^2 & (\text{if $q$ is even})\\[2mm]
\upsilon(T_{p-q,q}) - \frac{1}{4}(q^2-1) & (\text{if $q$ is odd})
\end{cases}
\]
for each coprime $p>q>0$. We use this formula to prove the following two claims.
(In the proof of the following claims, we denote by $\mathrm{RHS}$ (resp.\ $\mathrm{LHS}$) the right-hand side (resp.\ the left-hand side) of the desired equality for the case under consideration.) 
\begin{claim}
\label{claim:first-step}
For the case $(p,q)=(k+1,k)$, the equality holds.
\end{claim}

\begin{proof}
For the case $(p,q)=(k+1,k)$, it is clear that $t= h = 1$, and hence $\mathrm{RHS} = 1 - k$.
First, suppose that $k=2l$. Then the recursive formula shows
\begin{align*}
\mathrm{LHS} &= \upsilon (T_{2l+1,2l}) - \upsilon(T_{2l-1, 2l-2})\\[2mm]
  &= -l^2 + (l-1)^2 = 1 -2l = \mathrm{RHS}.
\end{align*}
Next, suppose that $k=2l+1$. Then
\begin{align*}
\mathrm{LHS} &= \upsilon (T_{2l+2,2l+1}) - \upsilon(T_{2l, 2l-1})\\[2mm]
  &= -(l+1)l + l(l-1) = -2l = \mathrm{RHS}. \qedhere
\end{align*}
\end{proof}

\begin{claim}
\label{claim:induction}
Suppose $p-q>1$. Set $p':=\max\{p-q, q\}$ and $q' := \min \{p-q, q\}$.
Then the equality for $(p',q')$ implies the equality for $(p,q)$.
\end{claim}

\begin{proof}
Let $(t,h)$ be the pair of integers with $0<t<p$, $0<h<q$ and $hp-tq = 1$. 
We first suppose that
$p'=p-q$ and $q' =q$. Then, the pair
$(t',h') := (t-h,h)$ satisfies 
\[
h'p' - t'q' = h(p-q) - (t-h)q = 1
\]
and $0 < h' < q'$.
Moreover, $p-q>1$ implies $0 < t' < p'$. Hence, by assumption, we have
\begin{align*}
\upsilon (T_{p-q,q}) - \upsilon(T_{(p-2t)-(q-2h), q -2h}) 
&= \upsilon (T_{p',q'}) - \upsilon(T_{p'-2t', q'-2h'}) \\[2mm]
&= \tfrac{1}{2} \big( 1 - h'(p'-t') - t'(q'-h') \big)\\[2mm]
&= \tfrac{1}{2} \big( 1 - h((p-t)-(q-h)) - (t-h)(q-h) \big)\\[2mm]
&= \mathrm{RHS} + h(q-h).
\end{align*}
If $p-2t>0$, then we also have $q-2h=q'-2h'>0$ and $(p-2t)-(q-2h)=p'-2t'>0$,
and hence 
\begin{align*}
\upsilon (T_{p-q,q}) - \upsilon(T_{(p-2t)-(q-2h), q -2h}) 
&=\mathrm{LHS}+\tfrac{1}{4}\big( q^2 - (q-2h)^2 \big)\\[2mm]
&=\mathrm{LHS} + h(q-h)
\end{align*}
regardless of the parity of $q$. Otherwise, we have $2t-p>0$, $2h-q>0$ and $(2t-p)- (2h-q) > 0$, and hence
\begin{align*}
\upsilon (T_{p-q,q}) - \upsilon(T_{(p-2t)-(q-2h), q -2h}) 
&=\upsilon (T_{p-q,q}) - \upsilon(T_{(2t-p)-(2h-q), 2h-q}) 
\\[2mm]
&=\mathrm{LHS}+\tfrac{1}{4}\big( q^2 - (2h-q)^2 \big)\\[2mm]
&=\mathrm{LHS} + h(q-h)
\end{align*}
regardless of the parity of $q$.

We next suppose that $p'=q$ and $q' =p-q$. 
Then, the pair
$(t',h') := (q-h, (p-q) - (t-h))$ satisfies 
\[
h'p' - t'q' = \big((p-q) - (t -h)\big)q - (q-h)(p-q) = 1
\]
and $0 < t' < p'$.
Moreover, $h'p' = t'q' +1 > 0$ and $p-q>1$ implies $0 < h' < q'$.
Hence, by assumption, we have
\begin{align*}
\upsilon (T_{q,p-q}) - \upsilon(T_{2h-q, (2t-p) -(2h-q)}) 
&= \upsilon (T_{p',q'}) - \upsilon(T_{p'-2t', q'-2h'}) \\[2mm]
&= \tfrac{1}{2} \big( 1 - h'(p'-t') - t'(q'-h') \big)\\[2mm]
&= \tfrac{1}{2} \big( 1 - ((p-t)-(q-h))h - (q-h)(t-h) \big) \\[2mm]
&= \mathrm{RHS} + h(q-h).
\end{align*}
The remaining part is proved similarly to the case of $(p',q')=(p-q,q)$.
\end{proof}
Now, let $p$ and $q$ be arbitrary coprime integers with $p>q>0$. Then, 
the Euclidean algorithm provides the sequence $\{ (p_i,q_i) \}_{i=1}^n$ of pairs 
of positive coprime integers such that 
\begin{itemize}
\item $(p_1, q_1) = (p,q)$,
\item $p_i - q_i > 1$ and $(p_{i+1}, q_{i+1}) = (\max\{p_i-q_i,q_i\}, \min\{p_i-q_i,q_i\})$ for any $1 \leq i < n$, and
\item $(p_n,q_n) = (k+1,k)$ for some $k>0$.
\end{itemize}
By \Cref{claim:first-step}, the equality for $(p_n,q_n)$ holds.
Then, the equality for $(p,q) = (p_1,q_1)$ follows from \Cref{claim:induction} by induction.
\end{proof}

\section{Computing methods using subsurfaces}

In this section, we prove \Cref{thm:subsurface} and \Cref{prop:positive-twist}.

\setcounter{section}{1}
\setcounter{thm}{4}
\begin{thm}
For a UST-invariant $f$ and knot $K$, 
let $F \subset B^4$ be a realizing surface for $(f,K)$.
If there exists a 4-ball $B \subset B^4$ such that $\partial B$ is transverse to $F$ and $K':=F \cap \partial B$ is a knot,  then $F' := F \cap B$ is a realizing surface for $(f,K')$.
\end{thm}
\setcounter{section}{4}
\setcounter{thm}{0}

\begin{proof}
The inequality $f(K') - \frac{\sigma(K')}{2} \leq b^-_2(M_{F'})$ immediately follows from
the inequality (\ref{eq:UST}) in \Cref{rem}. Let us prove the opposite inequality. 
For a given 4-ball $B \subset B^4$, we excise $B$ from $B^4$ to obtain a cobordism $C$ in $S^3 \times [0,1]$ from $K'$ to $K$. Moreover, by taking an arc $\alpha$ in $C$ connecting $K'$ to $K$ and excising a tubular neighborhood of $\alpha$, we have a surface $F''$ in $B^4$ with 
$\partial F'' = K \# K'^*$, where $K'^*$ denotes the orientation reversed mirror of $K'$ .
The construction of $F''$ implies
\[
b_2^-(M_{F''}) = b_2^-(M_C) = b_2^-(M_F) - b_2^-(M_{F'}).
\]
Applying the inequality (\ref{eq:UST}) to $F''$, we have
\begin{align*}
b_2^-(M_F) - b_2^-(M_{F'}) 
&\geq f(K\# K'^*) - \frac{\sigma(K\# K'^*)}{2} \\[2mm]
&= \left(f(K) - \frac{\sigma(K)}{2}\right) -\left( f(K') - \frac{\sigma(K')}{2} \right).
\end{align*}
Since $F$ is a realizing surface for $(f,K)$, this gives $f(K') - \frac{\sigma(K')}{2} \geq b^-_2(M_{F'})$.
\end{proof}

\setcounter{section}{1}
\setcounter{thm}{5}
\begin{prop}
For a UST-invariant $f$ and knot $K$, 
let $F$ be a ribbon surface which is a realizing surface for $(f,K)$. 
If a ribbon surface $F'$ with $\partial F' = K'$ is obtained from $F$ by a positive full-twist, then $F'$ is a realizing surface for $(f,K')$.
\end{prop}
\setcounter{section}{4}
\setcounter{thm}{0}

\begin{proof}
Let $\mathcal{M}_+ \subset S^3$ be the positively embedded standard M\"{o}bius band 
(i.e. $\partial \mathcal{M}_+$ is the unknot and $\sigma(M_{\mathcal{M}_+})=1$).
Then, the boundary connected sum $F \natural \mathcal{M}_+ \natural \mathcal{M}_+$
is also a realizing surface for $(f,K)$, and it is isotopic to a ribbon surface shown in the right-hand side of \Cref{fig:positive-twist2}, denoted by $F''$.
Moreover, an embedded surface in $B^4$ corresponding to $F''$ is isotopic to the composition of an embedded surface in $B^4$ corresponding to $F'$ and a cobordism given by a motion picture shown in \Cref{fig:positive-twist3}.
Therefore, \Cref{thm:subsurface} shows that $F'$ is a realizing surface for $(f,K')$.
\end{proof}

\begin{figure}[tbp]
\includegraphics[scale = 0.8]{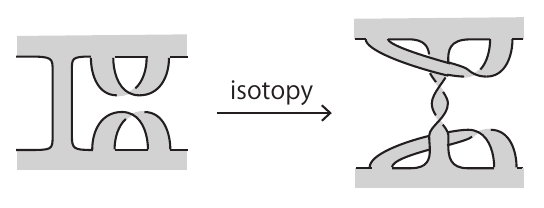}
\caption{\label{fig:positive-twist2} An isotopy of $F \natural \mathcal{M}_+ \natural \mathcal{M}_+$}
\end{figure}

\begin{figure}[tbp]
\includegraphics[scale = 0.8]{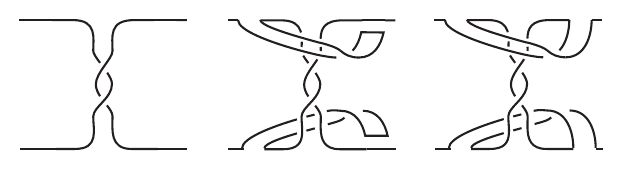}
\caption{\label{fig:positive-twist3} A cobordism from $K'$ to $K$}
\end{figure}

\section{Proof of \Cref{thm:topslice}}

Finally, we prove \Cref{thm:topslice}.

\setcounter{section}{1}
\setcounter{thm}{7}
\begin{thm}
Any unoriented slice-torus invariant does not factor through $\mC_{\Top}$.
\end{thm}
\setcounter{section}{5}
\setcounter{thm}{0}

\begin{proof}
It can be directly verified that the surface $F$ shown in the left-hand side of \Cref{fig:topslice} is a realizing surface for $T_{3,4} \# 5_2^*$. 
Consider a spatial graph $G$ embedded in $F$ as shown in the right-hand side of \Cref{fig:topslice}. Then it is not hard to check that the regular neighborhood of $G$ in $F$,
shown in the left-hand side of \Cref{fig:topslice2}, is isotopic to the standard Seifert surface $F'$ for $D_-(5_2)$, shown in the right-hand side of \Cref{fig:topslice2}.
Hence, it follows from \Cref{thm:subsurface} that $F'$ is a realizing surface for $(f,D_-(5_2))$.
In particualr, since $b^-_2(M_{F'})=1$ and $\sigma(D_-(5_2))=0$, we have
\[
f(D_-(5_2)) = b^-_2(M_F) + \frac{\sigma(D_-(5_2))}{2}= 1 \neq 0.
\]

Now we assume that $f$ factors through $\mC_{\Top}$.
Then, since $D_-(5_2)$ is topologically slice, we have $f(D_-(5_2))=0$,
which leads to a contradiction.
\end{proof}

\begin{figure}[tbp]
\includegraphics[scale = 0.8]{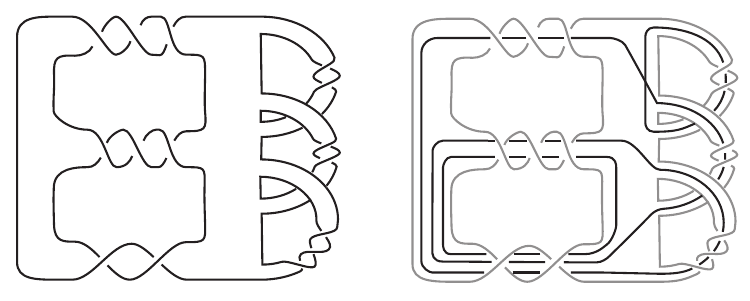}
\caption{\label{fig:topslice} A realizing surface $F$ and a spatial graph $G$}
\end{figure}

\begin{figure}[tbp]
\includegraphics[scale = 0.8]{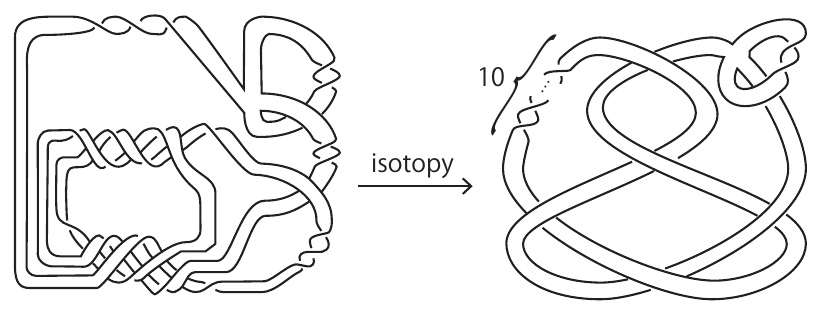}
\caption{\label{fig:topslice2} A regular neighborhood of $G$ in $F$ is isotopic to the standard Seifert surface for $D_-(5_2)$.}
\end{figure}

\bibliographystyle{plain}
\bibliography{tex}

\end{document}